\documentclass[12pt]{amsart}
\usepackage{amsfonts}
\usepackage{bbm}
\usepackage{amsfonts,amssymb,amsmath,amsthm}
\usepackage{url}
\usepackage{enumerate}
\usepackage{bbm}
\usepackage{amsfonts}
\usepackage{mathrsfs}
\usepackage{amsmath,amscd,amsthm,amsfonts}
\usepackage[all]{xy}
\usepackage[latin1]{inputenc}
\usepackage[pdftex, colorlinks, citecolor=red, backref=page]{hyperref}

\newtheorem{thm}{Theorem}[section]
\newtheorem{defi}[thm]{Definition}
\newtheorem{lem}[thm]{Lemma}
\newtheorem{coro}[thm]{Corollary}

\newtheorem{prop}[thm]{Proposition}
\newtheorem{rem}[thm]{Remark}

\theoremstyle{definition}
\numberwithin{equation}{section}

\newcommand{\Z}{Z\cap \overline{W^u(x, \delta)}}

\author {Wenda Zhang}
\address{College of Mathematics and Statistics, Chongqing Jiaotong University, Chongqing, China \ 400074}
\email{wendazhang951@aliyun.com}
\author{zhiqiang li$^*$}
\address{College of Mathematics and Statistics, Chongqing University, Chongqing, China \ 401331}
\thanks{$*$ \textit{Corresponding Author}}
\email{zqli@cqu.edu.cn}
\author{Xiankun Ren}
\address{College of Mathematics and Statistics, Chongqing University, Chongqing, China \ 401331}
\email{xkren@cqu.edu.cn}


\begin{document}

\title[Unstable topological pressure on saturated subsets]{Sub-additive unstable topological pressure on saturated subsets}

\begin{abstract}
In this paper, we continue our investigation on sub-additive pressures for $C^1$-smooth partially hyperbolic diffeomorphisms. Under the assumption of unstable almost product property, we show that the unstable Bowen topological pressure on the saturated set of a given non-empty compact connected set of invariant measures equals the infimum of the summation of unstable metric entropy and the corresponding \emph{Lyapunov exponent}, where the infimum is taken over all invariant measures inside the compact connected set above. Moreover, we also show that the unstable topological capacity pressure on the saturated sets above coincide with the unstable topological pressure of the whole system. 
\end{abstract}

\maketitle

\section{Introduction}

As a natural generalization of entropy, pressure is another key invariant for dynamical systems, which roughly measures the dynamical complexity on given (sub-additive) potential functions in both topological and measure theoretical settings. For references, one can see for example \cite{Ruelle},  \cite{Walters1}, \cite{Pesin1},  \cite{K. Falconer}, \cite{Bar}, \cite{Huang1},  \cite{Hu}, etc.

In recent years, the theory of unstable entropy and pressure for $C^1$-smooth partially hyperbolic diffeomorphisms are
intensively investigated. Unstable topological and metric entropy are more refined invariants, which rule out the complexity on central directions and focus on that on unstable directions.  They were introduced and studied in \cite{Hu1} by Hu, Hua, and Wu; moreover, they obtained the corresponding Shannon-McMillan-Breiman theorem, as well as the corresponding variational principle.
In \cite{Tian}, Tian and Wu generalize the result above with additional consideration of an arbitrary subset (not necessarily compact or invariant).
In \cite{Zhu2}, Hu, Wu, and Zhu investigated the unstable topological pressure for additive potentials. In \cite{Zhang} and \cite{Zhang1}, we worked on both unstable topological and measure theoretical pressure for $C^1$-smooth partially hyperbolic diffeomorphisms with sub-additive potentials, the expected variational principle was established.

In this paper, we we continue our investigation on sub-additive pressures for $C^1$-smooth partially hyperbolic diffeomorphisms. We mainly focus on unstable topological pressures on saturated sets in terms of Bowen's picture and the capacity picture. On one hand, we impose an analog of the almost product structure (see \cite{Sulli}) in current setting--the unstable almost product property, and then we are able to calculate the unstable Bowen's topological pressure on saturated sets. On the other hand, we also show that the unstable capacity pressure on the saturated sets coincide with the pressure of the whole system.

Throughout the paper $M$ is a finite dimensional, $C^1$-smooth, connected and compact Riemannian manifold
without boundary. Let $f : M \rightarrow M$ be a $C^{1}$-smooth partially hyperbolic diffeomorphism. We denote by $\mathcal{M}_{f}(M)$  the collection of $f$-invariant Borel probability measures on $M$. We always consider a sequence of continuous sub-additive potentials $\mathcal{G}=\{\log g_n\}_n$ of $f$ on $M$. 

The main results of this paper read as follows.

\begin{thm}\label{main1} 
Let $f:M\rightarrow M$ be a $C^1$-smooth partially hyperbolic diffeomorphism. Suppose $\mathcal{G}=\{\log g_n\}^{\infty}_{n=1}$ is a sequence of sub-additive potentials of $f$ such that the Lyapunov exponent $\mathcal{G}_*(\mu)$ is lower semi-continuous (w.\,r.\,t.\,$\mu$). If $(M, f)$ has unstable almost product property for some blow up function $g$, then for any non-empty compact connected set $K\subseteq \mathcal{M}_f(M)$, one has $$P^u_B(f, \mathcal
{G}, G_K)=\inf\{h^u_{\mu}(f)+\mathcal{G}_*(\mu)\mid \mu\in K\}.$$
\end{thm}

For the unstable capacity pressure on saturated sets, we have the following result.

\begin{thm}\label{main2}
Let $f:M\rightarrow M$ be a $C^1$-smooth partially hyperbolic diffeomorphism. Suppose $\mathcal{G}=\{\log g_n\}^{\infty}_{n=1}$ is a sequence of sub-additive potentials of $f$ satisfying the following condition: $$\limsup\limits_{\epsilon\rightarrow 0}\limsup\limits_{n\rightarrow\infty}\dfrac{\log \gamma_{n}(\mathcal{G},\epsilon)}{n}=0,$$ where $$\gamma_{n}(\mathcal{G},\epsilon)=\sup\{\dfrac{g_{n}(x)}{g_{n}(y)}\,|\,d(x,y)\leq \epsilon \}.$$ If $(M, f)$ enjoys the unstable almost product property for some blow up function $g$, then for any non-empty compact connected subset $K\subseteq \mathcal{M}_f(M)$, one has $$\underline{CP}^u(f, \mathcal{G}, G_K)=P^{u}(f,\mathcal{G}).$$
\end{thm}

The paper is organized as follows. In Section 2, we give necessary definitions and preliminaries. In Section 3, we calculate the unstable Bowen topological pressure on the saturated subsets $G_{K}$, and the proof of the result about unstable capacity pressure on $G_K$ is given in Section 4.

\section{Definitions and Preliminaries}

In the whole paper, $M$ is a finite dimensional, $C^1$-smooth, connected and compact Riemannian manifold
without boundary. Let $f : M \rightarrow M$ be a $C^{1}$-smooth partially hyperbolic diffeomorphism. 

\begin{defi}\label{pot}
A sequence of continuous functions $\mathcal{G} = \{\log g_{n}\}
_{n=1}^{\infty}$ on $M$ is
called a sequence of sub-additive potentials of $f$ if
$$\log g_{m+n}(x)\leq \log g_{n}(x)+\log g_{m}(f^{n}x), \,\text{for any}\; x\in M\,\text{and}\;m,n\in\mathbb{N}.$$
\end{defi}

\begin{rem}
For any $f$-invariant Borel probability measure $\mu$, set
$$\mathcal{G}_{*}(\mu)=\lim\limits_{n\rightarrow \infty}\frac{1}{n}\int\log g_{n}d\mu,$$
and $\mathcal{G}_{*}(\mu)$ is called the Lyapunov exponent of $\mathcal{G}$ with respect to $\mu$. The existence of this limit follows from a sub-additive argument. It takes values in
$[-\infty, +\infty)$. Moreover, the Sub-additive Ergodic Theorem (see \cite{Walters2}, Theorem 10.1) implies that for an ergodic measure $\mu$, one has $$\mathcal{G}_{*}(\mu)=\lim\limits_{n\rightarrow \infty}\frac{1}{n}\log g_{n}(x), \mu\,\text{-}\,a.e.\,x.$$
\end{rem}

Take $\epsilon_{0} > 0$
small. Let $\mathcal{P} = \mathcal{P}_{\epsilon_{0}}$ denote the set of finite Borel partitions $\alpha$ of $M$ whose elements
have diameters smaller than or equal to $\epsilon_{0}$, that is, $\mbox{diam}\;\alpha := \sup\{\mbox{diam}\;A : A \in
\alpha\} \leq \epsilon_{0}$. For each $\beta \in \mathcal{P}$ we can define a finer partition $\eta$  such that $\eta(x) =\beta(x) \cap W^{u}_{loc}(x)$ for each $x \in M$, where $W^{u}_{loc}(x)$ denotes the local unstable manifold
at $x$ whose size is greater than the diameter $\epsilon_{0}$ of $\beta$. Since $W^{u}$ is a continuous
foliation, $\eta$ is a measurable partition with respect to any Borel probability measure
on $M$.

Let $\mathcal{P}^{u}$ denote the set of partitions $\eta$ obtained in this way and \emph{subordinate
to unstable manifolds}. Here a partition $\eta$ of $M$ is said to be subordinate to unstable
manifolds of $f$ with respect to a measure $\mu$ if for $\mu$-almost every $x, \eta(x) \subset W^{u}(x)$
and contains an open neighborhood of $x$ in $W^{u}(x)$. It is clear that if $\alpha \in \mathcal{P}$
satisfies $\mu(\partial\alpha) = 0$, where $\partial\alpha := \cup_{A\in \alpha}\partial A$, then the corresponding $\eta$ given by
$\eta(x) = \alpha(x) \cap W^{u}_{loc}(x)$ is a partition subordinate to unstable manifolds of $f$.

Given any probability measure $\nu$ and any measurable partition $\eta$ of $M$, and denote by $\eta(x)$ the element of $\eta$ containing $x$.
The \emph{canonical system of conditional measures}
for $\nu$ and $\eta$ is a family of probability measures $\{\nu^{\eta}_{x} : x \in M\}$ with $\nu^{\eta}_{x}(\eta(x))= 1$,
such that for every measurable set $B \subseteq M, x \mapsto \nu^{\eta}_{x}(B)$ is measurable and
$$\nu(B) =\int_{X}\nu^{\eta}_{x}(B)d\nu(x).$$ This is also called the measure disintegration of $\nu$ over $\eta$. A classical result of Rokhlin (cf.\,\cite{Rohlin}) says that if $\eta$ is a measurable partition,
then there exists a system of conditional measures with respect to $\eta$. It is essentially
unique in the sense that two such systems coincide for sets with full $\nu$-measure. For
measurable partitions $\alpha$ and $\eta$, let
$$H_{\nu}(\alpha|\eta):=-\int_{M}\log\nu^{\eta}_{x}(\alpha(x))d\nu(x).$$
denote the conditional entropy of $\alpha$ for given $\eta$ with respect to $\nu$.

The unstable metric entropy in \cite{Hu1} is defined as follows.
\begin{defi} For any $\mu\in\mathcal{M}_f(M)$, any $\eta\in \mathcal{P}^u$, and any $\alpha\in\mathcal{P}$, define $$h_{\mu}(f, \alpha|\eta)=\limsup_{n\rightarrow\infty}\frac{1}{n}H_{\mu}(\alpha_{0}^{n-1}|\eta),$$ and $$h_{\mu}(f|\eta)=\sup\limits_{\alpha\in\mathcal{P}}h_{\mu}(f, \alpha|\eta).$$
The unstable metric entropy of $f$ is defined by
$$h^{u}_{\mu}(f)=\sup\limits_{\eta\in\mathcal{P}^{u}}h_{\mu}(f|\eta).$$
\end{defi}

\begin{rem}\label{rem}
In Theorem A of \cite{Hu1}, the authors proved that for any ergodic measure $\mu$, any $\alpha\in\mathcal{P},$ and any $\eta\in \mathcal{P}^u$,  one has $h^u_{\mu}(f)=h_{\mu}(f|\eta)=h_{\mu}(f, \alpha|\eta).$ Moreover, Corollary A.2 of \cite{Hu1} shows that the superior limit above for $h_{\mu}(f, \alpha|\eta)$ is actually the limit. 
\end{rem}

We denote by $d^u$ the metric induced by the Riemannian structure on the unstable manifold, and for any positive integer $n$, let $$d^u_n(x,y)=\max\limits_{0\leq j\leq n-1}d^u(f^j(x), f^j(y)).$$ For any $\epsilon>0$, the unstable $(n, \epsilon)$-Bowen ball around $x\in M$ is: $$B^u_n(x, \epsilon)=\{y\in W^u(x)\mid d^u_n(x, y)\leq\epsilon\}.$$
Let $W^u(x, \delta)$ be the open ball inside $W^u(x)$ centered at $x$ with radius $\delta$ with respect to the metric $d^u$.

\begin{defi} Let $g: \mathbb{N}\rightarrow \mathbb{N}$ be a non-decreasing unbounded function. $g$ is said to be blow up if the following hold: $$g(n)\leq n \;\text{and}\; \lim\limits_{n\to \infty}\dfrac{g(n)}{n}=0.$$
Set $\Lambda_n=\{0,\cdots, n-1\}$. For any $x\in M$ and any $\epsilon>0$, the unstable $g$-blow up of $B^u_n(x, \epsilon)$ is the following closed set: $$
\begin{aligned}B^u_n(g, x, \epsilon)&
:=\{y\in W^u(x)\mid \\
&\exists\, \Lambda\subseteq \Lambda_n \;\text{with}\; |\Lambda_n\setminus\Lambda|\leq g(n)\;\text{and} \; d^u(f^ix, f^iy)\leq \epsilon, \forall i\in \Lambda\}.
\end{aligned}$$
\end{defi}

\begin{defi}\label{un prod}
We say the dynamical system $(M, f)$ has the unstable almost product property for some blow up function $g$ if there is a non-increasing function $m:\mathbb{R}^+\rightarrow \mathbb{N}$ such that for any positive integer $K$, any $x_i\in M$, any $\epsilon_i>0$, and any $n_i>m(\epsilon_i)$, $1\leq i\leq K$, one has $$\bigcap\limits_{i=1}^{K}f^{-M_{i-1}}B^u_{n_{i}}(g, x_i, \epsilon_i)\neq\emptyset,$$ where $M_0=0, M_j=n_1+\cdots+n_j, 1\leq j\leq K-1.$
\end{defi}

\begin{defi} Let there be given a $C^1$-smooth partially hyperbolic diffeomorphism $f: M\rightarrow M$. We say that ergodic measures of $(M, f)$ are unstable entropy dense, if for any $\nu\in \mathcal{M}_f(M)$ and any neighborhood $F$of $\nu$,  whenever $h^*<h^u_\nu(f)$, one can find an ergodic measure $\mu\in F$ such that $h^*<h^u_\mu(f)$.
\end{defi}

Given a positive integer $n$, small tolerance $\rho>0$ and $\epsilon>0$, we say two points $y, z\in\overline{W^u(x, \delta)}$ are $(\rho, n, \epsilon)\,u$-separated if $$|\{0\leq i\leq n-1\mid d^u(f^iy, f^iz)>\epsilon\}|\geq n\rho.$$ A subset $S\subseteq \overline{W^u(x, \delta)}$ is called a $(\rho, n, \epsilon)\,u$-separated set if any two points in $S$ are $(\rho, n, \epsilon)\,u$-separated. For any $x\in M$, consider the empirical measure at $x$ as follows: $$\mathcal{E}_n(x):=\dfrac{1}{n}\sum\limits_{i=0}^{n-1}\delta_{f^ix}.$$ Suppose $\mu\in\mathcal{M}(M)$ and $F$ is a neighborhood of $\mu$. Set $$M_{n, F}:=\{x\in M\mid \mathcal{E}_n(x)\in F\}.$$ Denote by $N^u(F, \epsilon, n, x, \delta)$ the maximal cardinality of an $(n, \epsilon)\,u$-separated subset of $\overline{W^u(x, \delta)}\cap M_{n, F};$ denote by $N^u(F, \epsilon, n, \rho, x, \delta)$ the maximal cardinality of a $(\rho, n, \epsilon)\,u$-separated subset of $\overline{W^u(x, \delta)}\cap M_{n, F}.$

\begin{defi}\label{uni sep def}
A partially hyperbolic diffeomorphism $f: M\rightarrow M$ is said to have the unstable uniform separation property, if for any $\kappa>0$, there exist positive numbers $\rho^*$ and $\epsilon^*$ such that for any ergodic $\mu$, any neighborhood $F$ of $\mu$ in $\mathcal{M}(M)$, and any $\delta>0$, there is a positive integer $N$ such that if $n\geq N$, then $$\mathop{esssup}\limits_{\mu}N^u(F, \epsilon^*, n, \rho^*, x, \delta)\geq e^{n(h^u_{\mu}(f)-\kappa)}.$$
\end{defi}

\begin{rem}\label{sep}
The definition above appeared in Section 5 of \cite{Tian}, and they showed that $C^1$-smooth partially hyperbolic systems enjoy this property, see Theorem D there. In fact, they proved that there is a set $C$ with $\mu(C)>0$ such that for any $x\in C$, if $n\geq N$, one can construct a $(\rho^*, n, \epsilon^*)\,u$-separated set of $\overline{W^u(x, \delta)}\cap M_{n, F}$ with cardinality at least $e^{n(h^u_{\mu}(f)-\kappa)}$.
\end{rem}

For any $g\in C(M)$ and $\mu\in\mathcal{M}(M)$, set $$\langle g, \mu\rangle=\int_{M}gd\mu.$$ There exists a countable and separating set of continuous functions $\{g_1, \cdots, g_n,\cdots\}$ with $0\leq g_k(x)\leq 1$ such that the following metric induces the weak$^*$-topology on $\mathcal{M}(M)$: $$d(\mu, \nu):=\sum\limits_{k\geq 1}2^{-k}|\langle g_k, \mu-\nu\rangle|.$$ For any $x, y\in M$, define $D(x, y)=d(\delta_x, \delta_y),$ then it is equivalent to the original metric on $M$, so for convenience we just use this one  directly and still denote it by $d$.

Let $B(\mu, \zeta)$ be the closed ball centered at a measure $\mu$ with radius $\zeta$ in $\mathcal{M}(M)$.
\begin{lem}\label{measure-nbh} Suppose $(M, f)$ has the unstable almost product property for some blow up function $g$. Let $x_1,\cdots,x_k\in M$, $n_1\geq m(\epsilon_1), \cdots, n_k\geq m(\epsilon_k)$, $\zeta_1,\cdots, \zeta_k>0$, and $\epsilon_1>0, \cdots, \epsilon_k>0$ (small enough) be given. If we assume that $$\mathcal{E}_{n_i}(x_i)\in B(\mu_i, \zeta_i), \quad i=1,\cdots, k.$$ Then for any $y\in \bigcap\limits_{i=1}^kf^{-M_{i-1}}B^u_{n_i}(g, x_i, \epsilon_i)$ and any $\nu\in \mathcal{M}(M)$, one has $$d(\mathcal{E}_{M_k}(y), \nu)\leq \sum\limits_{i=1}^k\dfrac{n_i}{M_k}(\zeta_i'+d(\mu_i, \nu)),$$ where $M_i=n_1+\cdots+n_i$, and
$\zeta_i'=\zeta_i+\epsilon_i+g(n_i)/n_i.$
\end{lem}
\begin{proof}
First we have $$\mathcal{E}_{M_k}(y)=\dfrac{1}{M_k}\sum_{i=1}^{M_k-1}\delta_{f^iy}=\sum_{j=1}^{k}\dfrac{n_j}{M_k}\mathcal{E}_{n_j}(f^{M_{j-1}}y).$$ Note that for $\delta>0$ small enough, there exists a $C>1$ such that for any $x\in M$, one has $$d(y, z)\leq d^u(y, z)\leq Cd(y, z)$$ for any $y, z\in\overline{W^u(x, \delta)}$. Hence
$$
\begin{aligned}
d(\mathcal{E}_{n_j}(x_j), \mathcal{E}_{n_j}(f^{-M_{j-1}}y))&\leq \dfrac{1}{n_j}\sum_{m=0}^{n_j-1}d(f^mx_j, f^{M_{j-1}+m}y)\\
&\leq \dfrac{g(n_j)}{n_j}+\epsilon_j.
\end{aligned}$$ Then by the triangle inequality, one has $$d(\mathcal{E}_{M_k}(y), \nu)\leq \sum_{j=1}^k\dfrac{n_j}{M_k}d(\mathcal{E}_{n_j}(f^{M_{j-1}}y), \nu)\leq \sum_{j=1}^k\dfrac{n_j}{M_k}(\zeta'_j+d(\mu_j, \nu)).$$
\end{proof}

Finally we gather several definitions of unstable topological pressure, one can also see them in \cite{Zhang1}.

\begin{defi}\label{Bowen pres}
For any $s\in \mathbb{R}$, $\delta>0$, $N\in \mathbb{N}$, $\epsilon>0$, $x\in M$, and any $Z\subseteq M$, set $$M^u(\mathcal{G}, s, N, \epsilon, Z, \overline{W^u(x, \delta)}):=\inf\limits_{\Gamma}\{\sum\limits_{i}{\exp(-sn_i+\sup\limits_{y\in B^u_{n_i}(x_i, \epsilon)}\log g_{n_i}(y))}\},$$ where $\Gamma$ runs over all countable open covers $\Gamma=\{B_{n_i}^u(x_i, \epsilon)\}_{i\in I}$ of $Z\cap \overline{W^u(x, \delta)}$ with each $n_i\geq N$.

Let $$m^u(\mathcal{G}, s, \epsilon, Z, \overline{W^u(x, \delta)}):=\lim\limits_{N\to\infty}M^u(\mathcal{G}, s, N, \epsilon, Z, \overline{W^u(x, \delta)}),$$
$$\begin{aligned}
P^u_{B}(f, \mathcal{G}, \epsilon, Z, \overline{W^u(x, \delta)})&:=\inf\{s\mid m^u(\mathcal{G}, s, \epsilon, Z, \overline{W^u(x, \delta)})=0\},\\
&:=\sup\{s\mid m^u(\mathcal{G}, s, \epsilon, Z, \overline{W^u(x, \delta)})=\infty\},\\
\end{aligned}$$ and
$$P^u_{B}(f, \mathcal{G}, Z, \overline{W^u(x, \delta)}):=\liminf \limits_{\epsilon\to 0}P^u_B(f, \mathcal{G}, \epsilon, Z, \overline{W^u(x, \delta)}),$$ then define $$P^u_B(f, \mathcal{G}, Z):=\lim\limits_{\delta\to 0}\sup\limits_{x\in M}P^u_B(f, \mathcal{G}, Z, \overline{W^u(x, \delta)}).$$
We call $P^u_B(f, \mathcal{G}, Z)$ the unstable Bowen topological pressure of $f$ on the subset $Z$ w.\,r.\,t.\,$\mathcal{G}$, and denote by $h^u_{B}(f, Z)$ the corresponding Bowen unstable topological entropy when the potentials vanish.
\end{defi}

\begin{defi}\label{Capacity-Top-Pressure} For any positive integer $n$, any $\epsilon>0$, any subset $Z\subseteq M$, any $\delta>0$, and any $x\in M$, set
$$\Lambda^u(\mathcal{G}, n, \epsilon, Z, \overline{W^u(x, \delta)}):=\inf\limits_{\Gamma}\{\sum\limits_{i}\sup\limits_{y\in B^u_{n_i}(x_i, \epsilon)}g_n(y)\},$$ where $\Gamma$ runs over all open covers $\Gamma=\{B_{n_i}^u(x_i, \epsilon)\}_{i\in I}$ of $\Z$ with $n_i=n$ for all $i$.

Then define $$\underline{CP}^u(f, \mathcal{G}, \epsilon, Z, \overline{W^u(x, \delta)}):=\liminf\limits_{n\to\infty}\frac{1}{n}\log \Lambda^u(\mathcal{G}, n, \epsilon, Z, \overline{W^u(x, \delta)}),$$
$$\overline{CP}^u(f, \mathcal{G}, \epsilon, Z, \overline{W^u(x, \delta)}):=\limsup\limits_{n\to\infty}\frac{1}{n}\log \Lambda^u(\mathcal{G}, n, \epsilon, Z, \overline{W^u(x, \delta)}),$$
$$\underline{CP}^u(f, \mathcal{G}, Z, \overline{W^u(x, \delta)}):=\liminf\limits_{\epsilon\to0}\underline{CP}^u(f, \mathcal{G}, \epsilon, Z, \overline{W^u(x, \delta)}),$$
$$\overline{CP}^u(f, \mathcal{G}, Z, \overline{W^u(x, \delta)}):=\liminf\limits_{\epsilon\to0}\overline{CP}^u(f, \mathcal{G}, \epsilon, Z, \overline{W^u(x, \delta)}).$$ Then the lower and upper capacity unstable topological pressures of $f$ on $Z$ w.\,r.\,t.\,$\mathcal{G}$ are defined by $$\underline{CP}^u(f, \mathcal{G}, Z):=\lim\limits_{\delta\to 0}\sup\limits_{x\in M}\underline{CP}^u(f, \mathcal{G}, Z, \overline{W^u(x, \delta)}),$$ and $$\overline{CP}^u(f, \mathcal{G}, Z):=\lim\limits_{\delta\to0}\sup\limits_{x\in M}\overline{CP}^u(f, \mathcal{G}, Z, \overline{W^u(x, \delta)}).$$

\end{defi}

\section{Unstable Pressure on Saturated Sets}

\begin{lem}\label{Inf entropy}
Let $(X, f)$ be a topological dynamical system. Suppose $\mu$ is an invariant Borel probability measure of $(X, f)$, and $\xi$ is a finite measurable partition of $X$. Then one has $$H_{\mu}(\xi)\leq \log 2+\mu(A)H_{\mu|_{A}}(\xi)+\mu(X\setminus A)H_{\mu|_{x\setminus A}}(\xi),$$ where $A$ is any Borel subset of $X$ and $\mu|_{(\cdot)}$ denotes the restriction of $\mu$.
\end{lem}

The following result follows from a standard combinatorial argument, see also Lemma 2.4 of \cite{Sulli2} for a proof.
\begin{lem}\label{est}
Let $A$ be a finite set. For any word $\omega\in A^{\Lambda_n}$, if $0\leq \delta\leq \frac{|A|-1}{|A|}$, then $$|\{v\in A^{\Lambda_n}\mid d^H_n(v, \omega)\leq \delta n\}|\leq e^{n\varphi(\delta)}(|A|-1)^{n\delta},$$ where $\varphi(\delta)=-\delta\log \delta-(1-\delta)\log (1-\delta).$
\end{lem}

The following proposition is a natural generalization of the unstable uniform separation property for unstable pressure. It originates from Theorem 3.1 of \cite{Sulli}.
\begin{prop}\label{first half for ergodic}
Let $f:M\rightarrow M$ be a $C^1$-smooth partially hyperbolic diffeomorphism  and $\mathcal{G}=\{\log g_n\}^{\infty}_{n=1}$ be a sequence of sub-additive potentials of $f$ on $M$. For any $\kappa>0$, there exists positive numbers $\rho^*>0$ and $\epsilon^*>0$ such that for any $\mu\in\mathcal{M}^e_{f}(M)$, any neighborhood $F$ of $\mu$ in $\mathcal{M}(M)$, and any $\delta>0$, there exists a subset $C$ of $M$ with $\mu(C)>0$ and a positive integer $n^*_{F, \mu, \kappa}$ such that for any $x\in C$, if $n\geq n^*_{F, \mu, \kappa}$, then one can find a $(n, \epsilon^*, \rho^*)$\,u-separated subset $\Gamma$ of $\overline{W^u(x, \delta)}\cap M_{n, F}$ such that $$\sum\limits_{y\in \Gamma}g_n(y)\geq e^{n\left(h^u_{\mu}(f)+\mathcal{G}_{*}(\mu)-\kappa\right)}.$$
\end{prop}

\begin{proof}
By Remark \ref{rem}, for any $\mu\in \mathcal{M}^e_{f}(M)$, any $\alpha\in\mathcal{P}_{\epsilon/2}$ where $\epsilon>0$ is small, and any $\eta\in \mathcal{P}^u$, one has $$h^u_{\mu}(f)=h_{\mu}(f, \alpha|\eta).$$

Assume $\alpha=\{A_1, \cdots, A_m\}$. For any $\kappa>0$, choose $0<\kappa'<\kappa$ and $\rho^*>0$ with $2\kappa'+\rho^*<1/2$, and also satisfy $$\phi(\rho^*+2\kappa')+(\rho^*+2\kappa')\log(2m-1)<\kappa-\kappa',$$ where $\phi(x)=-x\log x-(1-x)\log (1-x).$

For any $\nu\in \mathcal{M}_f(M)$, we shall construct a neighborhood $W_{\nu}$ of $\nu$ in $\mathcal{M}(M)$ as follows. Since $\nu$ is regular, for each $1\leq i\leq m$, choose a compact subset $B_i\subseteq A_i$ with $$\nu(A_i\setminus B_i)<\dfrac{\kappa'}{4m\log 2m}.$$ Take $\epsilon^*>0$ such that $d(y, z)>2\epsilon^*$ if $y\in B_i, z\in B_j, i\neq j$. Choose a positive integer $n_1$ such that for any $n>n_1$, one has $$\dfrac{\kappa'}{4\log m}\geq \dfrac{\log2}{n\log 2m}.$$ For each $1\leq i\leq m$, take an open neighborhood $U_i\supset B_i$ with $diam(U_i)<\epsilon$ (where $\epsilon$ is as in the beginning), and $d(y, z)>\epsilon^*$ whenever $y\in U_i, z\in U_j, i\neq j.$ Set $K=M\setminus \bigcup\limits_{i}U_i$, then $K$ is closed. So the characteristic function $\chi_{K}$ is upper semi-continuous. Define $$W_{\nu}=\{\lambda\in \mathcal{M}(M)\mid
\int_{M}\chi_{K}d\lambda\leq \int_M\chi_Kd\nu+\dfrac{\kappa'}{4\log 2m}\}.$$ Then it is enough to prove the statement for any ergodic measure $\mu\in W_{\nu}\cap \mathcal{M}_f(M)$. 

Let $\alpha'$ be the finite partition of $M$ which consists of all $U_i, 1\leq i\leq m$ and all $A_i\setminus\bigcup_jU_j, 1\leq i\leq m$, note that $diam(\alpha')<\epsilon$ as above. Label the atoms of $\bigvee\limits_{i=0}^{n-1}f^{-i}\alpha'$ by words of length n over an alphabet $A$ of at most $2m$ letters. The letter $1, 2, ..., m$ indicate the member $U_1, ..., U_m$ of $\alpha'$, and the others indicate all non-empty $A_1\setminus \bigcup_jU_j,..., A_m\setminus \bigcup_jU_j$. Define $\omega: M\rightarrow A^{\Lambda_n}$ as follows: $$\omega(x)(i)=j\;\text{if}\; f^i(x)\;\text{is in the atom of}\;\alpha'\;\text{labeled by}\, j.$$

Now we prove the result for any ergodic measure $\mu\in W_{\nu}$, any neighborhood $F$ of $\mu$, and any $\delta>0$. Choose $\eta\in\mathcal{P}^u$ such that $\eta(x)\subseteq W^u(x, \delta)$ for almost every $x$. Since $\alpha'\in\mathcal{P}_{\epsilon}$, based on Remark \ref{rem}, one can find a positive integer $n_2$ such that if $n\geq n_2$, then $$H_{\mu}((\alpha')^{n-1}_{0}\mid\eta)>n(h^u_{\mu}(f)-\kappa'/4),\quad\quad(1).$$ By the definition of $W_{\nu}$, one has $$\mu(K)\leq \nu(K)+\dfrac{\kappa'}{4\log 2m},$$ while $$\nu(K)= \nu(M\setminus \bigcup\limits_{i}B_i)\leq \nu(\bigcup\limits_{i}A_i\setminus B_i).$$ Hence $$\mu(K)\leq \sum\limits_{i=1}^m\nu(A_i\setminus B_i)+\dfrac{\kappa'}{4\log 2m}\leq \dfrac{\kappa'}{2\log 2m}.$$

Set $$Y_n=\{x\in M\mid |\{i\in \Lambda_n\mid \omega(x)(i)>m)\}|\leq n\kappa'\},$$ then by Birkhoff's Ergodic Theorem, one has $\lim\limits_{n\to\infty}\mu(Y_n)=1$. Note that we also have $\lim\limits_{n\to\infty}\mu(M_{n, F})=1$.
On the other hand, by the Sub-additive Ergodic Theorem, for almost all $y\in M$, there exists an $N(y, \kappa')$ such that if $n>N(y, \kappa')$, one has $$\dfrac{1}{n}\log g_n(y)>\mathcal{G}_{*}(\mu)-\kappa'/2,\quad\quad \mu\text{-}a.e.\,y.$$ Put $$E_n=\{y\in M\mid N(y, \kappa')\leq n\},$$ then $E_n\subseteq E_{n+1}$ and $\lim\limits_{n\to\infty}\mu(E_n)=1.$
Set $V_n=M_{n,F}\cap Y_n\cap E_n$, then there exists an $n_3$ such that if $n\geq n_3$, then $$\mu(M\setminus V_n)<\dfrac{\kappa'}{4\log 2m}-\dfrac{\log 2}{n\log2m}.$$

Let $n^*_{F, \mu, \kappa}=\max\{n_1, n_2, n_3\}$, and for $n\geq n^*_{F, \mu, \kappa}$, consider the partition $\beta_n=\{V_n, M\setminus V_n\}$. Recall the measure disintegration of $\mu$ over $\eta\in\mathcal{P}^u$: $$\mu(B)=\int_M\mu^\eta_x(B)d\mu(x)\;\text{for every measurable set}\;B.$$
We claim that there exist a subset $C$ of $M$ with $\mu(C)>0$ such that for any $x\in C$, one has $$H_{\mu^{\eta}_{x}|_{V_n}}((\alpha')^{n-1}_{0})\geq n(h^u_{\mu}(f)-\kappa'/2),   \quad\quad(2).$$
Otherwise for $\mu$-a.e.\,$x\in M$, one has the contrary of the inequality above. By the definition of conditional measure, one has $$H_{\mu}((\alpha')^{n-1}_0\mid\eta)=\int_{M}H_{\mu^{\eta}_{x}}((\alpha')^{n-1}_{0})d\mu(x),$$ applying Lemma \ref{Inf entropy} for $\mu^{\eta}_x|_{V_n}$, one gets  $$
\begin{aligned}
&H_{\mu}((\alpha')^{n-1}_0\mid\eta)\\
\leq& \log 2+\int_M\left(\mu^{\eta}_x(V_n)H_{\mu^{\eta}_x|_{V_n}}((\alpha')^{n-1}_0)+\mu^{\eta}_x(M\setminus V_n)H_{\mu^{\eta}_x|_{M\setminus V_n}}((\alpha')^{n-1}_0)\right)d\mu(x)\\
\leq&\log 2+\int_M\left(\mu^{\eta}_x(V_n)n(h^u_{\mu}(f)-\kappa'/2)+\mu^{\eta}_x(M\setminus V_n)n\log 2m\right)d\mu(x)\\
\leq&\log 2+\mu(V_n)n(h^u_{\mu}(f)-\kappa'/2)+\mu(M\setminus V_n)n\log 2m.
\end{aligned}$$ By (1) before, we have $$n(h^u_{\mu}(f)-\kappa'/4)<\log 2+\mu(V_n)n(h^u_{\mu}(f)-\kappa'/2)+\mu(M\setminus V_n)n\log 2m.$$ While $$\mu(M\setminus V_n)<\dfrac{\kappa'}{4\log 2m}-\dfrac{\log 2}{n\log2m},$$ then one gets a contradiction.

For any $x\in C$, let $E_n(x)$ denote the image of $V_n\cap \eta(x)$ by the map $\omega$. By (2), we have $$|E_n(x)|\geq e^{n(h^u_{\mu}(f)-\kappa'/2)}.$$ Then the Hamming distance $d^{H}_n(\omega, \omega')$ of two words $\omega$ and $\omega'$ in $A^{\Lambda_n}$ is the number of different letters of them. Suppose $E'_n(x)\subseteq E_n(x)$ is a set of maximal cardinality such that $d^{H}_n(\omega, \omega')>n(2\kappa'+\rho^*)$ for any two different words in $E'_n(x)$. Let $\Gamma_n(x)$ be the set obtained by taking exactly one point from the atom of $(\alpha')^{n-1}_0$ which is labeled by a word in $E'_n(x)$. It is easy to see that $\Gamma_n(x)\subseteq M_{n, F}\cap \eta(x)$ is a $(\rho^*, n, \epsilon^*)\,u$-separated set. The maximal cardinality of $E'_n(x)$ implies that for any $\omega\in E_n(x)$, there exists $\omega'\in E'_n(x)$ such that $d^{H}_n(\omega, \omega')\leq n(2\kappa'+\rho^*)$. Note that for any $\omega\in A^{\Lambda_n}$, by Lemma \ref{est}, one has $$|\{\omega'\in A^{\Lambda_n}\mid d^H_n(\omega', \omega)\leq n(2\kappa'+\rho^*)\}|\leq e^{n\phi(2\kappa'+\rho^*)}(|A|-1)^{n(2\kappa'+\rho^*)}.$$ Hence by the choice of $\kappa'$ and $\rho^*$, one has $$
\begin{aligned}
|\Gamma_n(x)|=&|E'_n(x)|\\
\geq&\dfrac{e^{n(h^u_{\mu}(f)-\kappa'/2)}}{e^{n\phi(2\kappa'+\rho^*)}(2m-1)^{n(2\kappa'+\rho^*)}}\\
\geq&e^{n(h^u_{\mu}(f)-\kappa/2)}.\\
\end{aligned}$$ 
Hence $$\sum_{y\in\Gamma_n(x)}g_n(y)\geq |\Gamma_n(x)|e^{n(\mathcal{G}_{*}(\mu)-\kappa/2)}\geq e^{n(h^u_{\mu}(f)+\mathcal{G}_{*}(\mu)-\kappa)}.$$
\end{proof}

\begin{coro}\label{inv}
Let $\mathcal{G}=\{\log g_n\}_n$ be a sequence of sub-additive potentials of $f$. Suppose $\mathcal{G}_{*}(\mu)$ is lower semi-continuous with respect to $\mu$, and the ergodic measures of $(M, f)$ are unstable entropy dense. Then for any $\kappa>0$, there exists a $\rho^*>0$ and $\epsilon^*>0$ such that for any $\mu\in\mathcal{M}_{f}(M)$, any neighborhood $F$ of $\mu$ in $\mathcal{M}(M)$ and any $\delta>0$, there exists a set $C\subseteq M$ and a positive integer $n^*_{F, \mu, \kappa}$ such that for any $x\in C$, if $n\geq n^*_{F, \mu, \kappa}$, then $$P^u(F, \mathcal{G}, n, \epsilon^*,  \rho^*, x, \delta)\geq e^{n(h^u_{\mu}(f)+\mathcal{G}_{*}(\mu)-\kappa)}.$$
\end{coro}\begin{proof}
If $\mu$ is ergodic, then it is done by Proposition \ref{first half for ergodic}. So we prove it for non-ergodic $\mu$. Since $\mathcal{G}_*(\mu)$ is lower semi-continuous, for any $\kappa>0$, there exists a neighborhood of $\mu$, say $F_1$, such that for any $\nu\in F_1$, one has $$\mathcal{G}_*(\nu)>\mathcal{G}_*(\mu)-\kappa/3.$$
On the other hand, by assumption there is an ergodic $\mu'\in F\cap F_1$ such that $$h^u_{\mu'}(f)>h^u_{\mu}(f)-\kappa/3.$$

By Proposition \ref{first half for ergodic}, there exist $\rho^*>0$ and $\epsilon^*>0$ such that for any $\delta>0$, there is a subset $C$ with $\mu'(C)>0$ and a positive integer $N$ such that if $n\geq N$, then $$P^u(F_1\cap F, \mathcal{G}, n, \epsilon^*,  \rho^*, x, \delta)\geq e^{n(h^u_{\mu'}(f)+\mathcal{G}_{*}(\mu')-\kappa/3)}\quad\quad \forall x\in C.$$ Then $$P^u(F, \mathcal{G}, n, \epsilon^*,  \rho^*, x, \delta)\geq e^{n(h^u_{\mu}(f)+\mathcal{G}_{*}(\mu)-\kappa)}, \quad\quad \forall x\in C.$$
\end{proof}

\begin{rem}
It is well known that $\mathcal{G}_*(\mu)$ is upper semi-continuous, so the assumption in the corollary above really refers to the case where $\mathcal{G}_*(\mu)$ is continuous. We put it in this way here just to indicate we only need the lower semi-continuity.
\end{rem}

X. Tian showed the unstable uniform separation property in \cite{Tian}; actually the unstable almost product property implies that  the statement could pass to general invariant measures as follows.

\begin{thm}\label{uni sep}
Let $f: M\rightarrow M$ be a $C^1$-smooth partially hyperbolic diffeomorphism and suppose $(M, f)$ satisfies unstable almost product property for some blow up function $g$. Then for any $\kappa>0$, there exists a $\rho^*>0$ and an $\epsilon^*>0$ satisfying that for any $\mu\in \mathcal{M}_f(M)$, any neighborhood $F\subseteq \mathcal{M}(M)$ of $\mu$, and any $\delta>0$, there is a subset $C\subseteq M$ and a positive integer $n^*_{F,\,\mu,\,\kappa}$ such that for any $x\in C$, if $n\geq n^*_{F,\,\mu,\,\kappa}$, then there exists an $(\rho^*, n, \epsilon^*)\,u$-separated set of $\overline{W^u(x, \delta)}\cap M_{n, F}$ such that its cardinality is greater than or equal to $e^{n(h^u_{\mu}(f)-\kappa)}.$
\end{thm}
\begin{proof} If $\mu$ is ergodic, then the conclusion is verified by the first part of Theorem D in \cite{Tian}.

We only need to prove the case when $\mu\in \mathcal{M}_f(M)$ is not ergodic. By ergodic decomposition, we have $$h^{u}_{\mu}(f)=\int_{\mathcal{M}^{e}_f(M)}h^{u}_{\nu}(f)d\tau(\nu),$$ where $\tau$ is a probability measure supported on $\mathcal{M}^{e}_f(M)$. Then there exist $r$ numbers of $a_i\in (0,1)$ with $\sum\limits^{r}_{i=1}a_i=1$ and ergodic measures $\nu_i$ such that $$\nu:=\sum\limits^{r}_{i=1}a_i\nu_i\in F\;\;\text{and}\;\;h^{u}_{\nu}(f)\geq h^{u}_{\mu}(f)-\frac{\kappa}{3}.$$ Without loss of generality, we may assume $r=2$ and $a_1=a_2=\frac{1}{2}$. The proof for the general case is very similar. Take two neighborhoods $F_i$ of $\nu_i$ such that $\frac{1}{2}\mu_1+\frac{1}{2}\mu_2\in F$ for any $\mu_i\in F_i,\,i=1,2.$ By Theorem D in \cite{Tian} (actually its proof, see Remark \ref{sep}), for $\nu_i\in \mathcal{M}^{e}_f(M)$, $i=1, 2$, any $\delta_1>0$ and $\delta=\delta_2$ above, there exists a set $C_i\subset M$ with $\nu_i(C_i)>0$ and $n_i\in \mathbb{N}$ such that for any $x_i\in C_i$ and $n> n_i$, one can construct a $(2\rho^{*},n,2\epsilon^{*})\,u$-separated set $\Gamma_i$ of $\overline{W^u(x_i,\delta_i)}\cap M_{n,F_i}$ with cardinality $|\Gamma_i|\geq e^{n(h^{u}_{\nu_i}(f)-\frac{\kappa}{3})}$. Take $n^*\geq \max\{n_1,n_2\}$ and also satisfies \begin{equation}\label{g-almost} n^*>m(\epsilon),\,2\rho^*n^*>4g(n^*)+1\;\;\text{and}\;\;\frac{g(n^*)}{n^*}<\epsilon, \end{equation} where $m(\epsilon)$ is the data in the definition of unstable almost product property and $\epsilon<\frac{\epsilon^*}{4}$.

Next we will define the sequences $\{n'_j\},\,\{\epsilon'_j\},\,\{C'_j\},\,\{F'_j\}$ and $\{\Gamma'_j\}$ by setting $$n'_j:=n^*,\;\epsilon'_j:=\epsilon,\;C'_j:=C_{(j\,\text{mod}\,2)+1},\;F'_j:=F_{(j\,\text{mod}\,2)+1},\;\Gamma'_j:=\Gamma_{(j\,\text{mod}\,2)+1},\;$$ $\delta'_j:=\delta_{(j\,\text{mod}\,2)+1}$ and $x'_j:=x_{(j\,\text{mod}\,2)+1}.$ Set$$G_k:=\bigcap\limits^{k}_{j=1}\left(\bigcup\limits_{y_j\in \Gamma'_j}f^{-M_{j-1}}B_{n'_j}^{u}(g, y_j, \epsilon'_j)\right),$$ where $M_{j-1}:=\sum\limits^{j-1}_{l=1}n'_l=(j-1)n^*, j\geq2$, and $M_0:=0$.
Then $G_k$ is a non-empty closed set. We can label each set obtained by developing this formula by the branches of a labeled tree of height $k$, of which a branch is labeled by $(y_1,\cdots,y_k)$ with $y_j\in \Gamma'_j$. By unstable almost product property for $g$, the set $$G_{(y_1,\cdots,y_k)}=\bigcap\limits^{k}_{j=1}\left(f^{-M_{j-1}}B_{n'_j}^{u}(g,y_j,\epsilon'_j)\right)$$ is contained in $W^{u}(y_1)$. Since $\Gamma'_1\subseteq\overline{W^u(x'_1,\delta'_{1})}\cap M_{n^*,F'_1}$, one has $W^u(y_1)=W^u(x'_1)$ and  $G_k\subseteq W^u(x'_1, \delta'_1)$.

Set $G:=\bigcap\limits^{\infty}_{k=1}G_k$. We claim that $G$ is a closed set, which is the disjoint union of non-empty closed sets of form $G_{(y_1,\cdots, y_k,\cdots)}$, which are called branch sets. More precisely, suppose $G_{(y_1,\cdots, y_k, \cdots)}$ and $G_{(z_1,\cdots, z_2, \cdots)}$ are two branch sets labeled by two different sequences $(y_1,y_2, ,\cdots)$ and $(z_1,z_2,\cdots)$, respectively. Take $y_j, z_j\in\Gamma'_j$ with $y_j\neq z_j$. For any $y\in B^{u}_{n^*}(g, y_j, \epsilon'_j)$ and $z\in B^{u}_{n^*}(g, z_j, \epsilon'_j)$, since $y_j$ and $z_j$ are $(2\rho^*,n^*,2\epsilon^*)\,u$-separated and (\ref{g-almost}) holds, one can find an $0\leq m\leq n^*-1$ satisfying $$d^u(f^{m}y_j, f^{m}z_j)\geq 2\epsilon^*,\;\;d^u(f^{m}y_j, f^{m}y)\leq\epsilon'_j\;\text{and}\;d^u(f^{m}z_j, f^{m}z)\leq\epsilon'_j.$$ By the triangle inequality, one has
$$\begin{aligned}d^u(f^{m}y, f^{m}z)&\geq d^u(f^{m}y_j, f^{m}z_j)-d^u(f^{m}y_j, f^{m}y)-d^u(f^{m}z, f^{m}z_j)\\&\geq\epsilon^*-2\epsilon>\epsilon^*.\\
\end{aligned}$$
Moreover, by (\ref{g-almost}) the number of such $m$ is more than $2\rho^*n^*-g(n^*)\geq 2\rho^*n^*-\epsilon n^*$. In a word, two different sequences label two different branch sets, and any two points in different branch sets are $(\rho^*,n,\epsilon^*)\,u$-separated provided that $\epsilon<\rho^*$ is small enough and $n>n^*$ is large enough.

To get the desired result, we also need to show that for any $y\in G$, one has $\mathcal{E}_n(y)\in F$. For any $n$ large enough, suppose $m=[\frac{n}{2 n^*}]$. Then $$\mathcal{E}_{n}(y)=\frac{2mn^*}{n}\mathcal{E}_{2mn^*}(y)+\frac{n-2mn^*}{n}\mathcal{E}_{n-2mn^*}(f^{2mn^{*}}y),$$ and 
$$\begin{aligned}
&d\left(\mathcal{E}_{n}(y), \mathcal{E}_{2mn^*}(y)\right)\\
=&\|\frac{n-2mn^*}{n}\left(\mathcal{E}_{n-2mn^*}(f^{2mn^{*}}y)-\mathcal{E}_{2mn^*}(y)\right)\|\leq \frac{2(n-2mn^*)}{n},
\end{aligned}$$ which can be arbitrary small as $n$ goes large enough. 

Note that $$\mathcal{E}_{2mn^*}(y)=\sum\limits^{2m-1}_{j=0}\frac{n^*}{2mn^*}\mathcal{E}_{n^*}(f^{jn^*}y).$$ Suppose the first $2m$ labels of $y$ are $(y_1, \cdots, y_{2m})$, which implies $y\in \bigcap\limits^{2m}_{j=1}\left(f^{-(j-1)n^*}B_{n^*}^{u}(g;y_j,\epsilon'_j)\right)$. Then for any $0\leq j\leq 2m-1$, one has $$\begin{aligned}
d\left(\mathcal{E}_{n^*}(f^{jn^*}y),\mathcal{E}_{n^*}(y_j)\right)&\leq\frac{1}{n^*}\sum\limits_{m=0}^{n^*-1}d^u\left(f^m(f^{jn^*}y),f^m(y^j)\right)\\&\leq\frac{g(n^*)}{n^*}+\epsilon\frac{n^*-g(n^*)}{n^*}\leq 2\epsilon.
\end{aligned}$$
So $\mathcal{E}_{n^*}(f^{jn^*}y)\in B\left(\mathcal{E}_{n^*}(y_j),2\epsilon\right)\subseteq F'_j$. On the other hand, if we adjust the summation order of $\mathcal{E}_{2mn^*}(y)$ as follows,$$\begin{aligned}
\mathcal{E}_{2mn^*}(y)&=\frac{1}{2m}\left(\mathcal{E}_{n^*}(y)+\mathcal{E}_{n^*}(f^{2n^*}y)+\cdots+\mathcal{E}_{n^*}(f^{2(m-1)n^*}y)\right)\\
&+\frac{1}{2m}\left(\mathcal{E}_{n^*}(f^{n^*}y)+\mathcal{E}_{n^*}(f^{3n^*}y)+\cdots+\mathcal{E}_{n^*}(f^{(2m-1)n^*}y)\right)\\
&=\frac{1}{2}\left(\frac{1}{m}\left(\mathcal{E}_{n^*}(y)+\mathcal{E}_{n^*}(f^{2n^*}y)+\cdots+\mathcal{E}_{n^*}(f^{2(m-1)n^*}y)\right)\right)\\
&+\frac{1}{2}\left(\frac{1}{m}\left(\mathcal{E}_{n^*}(f^{n^*}y)+\mathcal{E}_{n^*}(f^{3n^*}y)+\cdots+\mathcal{E}_{n^*}(f^{(2m-1)n^*}y)\right)\right),
\end{aligned}$$
then $\mathcal{E}_{2mn^*}(y)\in F$. If $n$ is large enough, then $(n-2mn^*)/n$ is small enough , which implies $\mathcal{E}_n(y)\in F.$

Set $C=C_2$, for any point $x'_1(=x_2)\in C$, we have found a $(\rho^*, n, \epsilon^*)\,u$-separated set 
of $\overline{W^u(x'_1, \delta)}\cap M_{n,F}$, whose cardinality is greater than or equal to $\prod\limits^{2m}_{j=1}|\Gamma'_j|$. Furthermore,
$$\begin{aligned}
\prod\limits^{2m}_{j=1}|\Gamma'_j|&\geq \left(e^{n^*(h^{u}_{\nu_1}(f)-\frac{\kappa}{3})}e^{n^*(h^{u}_{\nu_2}(f)-\frac{\kappa}{3})}\right)^{m}\\
&\geq e^{2n^*m\left(\frac{1}{2}h^{u}_{\nu_1}(f)+\frac{1}{2}h^{u}_{\nu_2}(f)-\frac{\kappa}{3}\right)}\\
&\geq e^{2n^*m\left(h^{u}_{\nu}(f)-\frac{\kappa}{3}\right)}\geq e^{(n-2n^*)\left(h^{u}_{\mu}(f)-\frac{2\kappa}{3}\right)}\geq e^{n\left(h^{u}_{\mu}(f)-\kappa\right)}.
\end{aligned}$$
\end{proof}

In \cite{Tian}, the authors introduced the concept of unstable upper capacity entropy on an arbitrary subset $Z$ of $M$ with respect to $f$:  $$h^{u}_{UC}(f,Z)=\lim\limits_{\delta\rightarrow 0}\sup\limits_{x\in M}\lim\limits_{\epsilon\rightarrow0}h_{UC}(f,\overline{W^{u}(x,\delta)}\cap Z),$$ where 
$$h_{UC}(f,\overline{W^{u}(x,\delta)}\cap Z)=\limsup\limits_{n\rightarrow\infty}\dfrac{1}{n}\log N^{u}(Z, \epsilon, n, x, \delta),$$ and $N^{u}(Z, \epsilon, n, x, \delta)$ is the maximal cardinality of an $(n,\epsilon) \,u$-separated set of $\overline{W^{u}(x,\delta)}\cap Z$, see Proposition 2.4 and Lemma 2.5 of \cite{Tian}. They also obtained the following variational principle.
\begin{prop}\label{VP}(Theorem 2.11 in \cite{Tian})
Let $f: M\rightarrow M$ be a $C^1$-smooth partially hyperbolic diffeomorphism and $Z\subseteq M$ be a compact $f$-invariant  subset. Then
$$\begin{aligned}h^{u}_{UC}(f,Z)&=\sup\{h^{u}_{\mu}(f)\mid\mu\in\mathcal{M}_f(M)\,\text{and}\; \mu(Z)=1\}\\
&=\sup\{h^{u}_{\mu}(f)\mid\mu\in\mathcal{M}^{e}_f(M)\,\text{and}\; \mu(Z)=1\}.\end{aligned}.$$
\end{prop}

\begin{rem}
Theorem \ref{uni sep} is an improved analog of Proposition 2.1 in \cite{Sulli2} in the following sense: 1) the current uniformness for measures follows from the (unstable) uniform separation property of Definition \ref{uni sep def}; 2) (unstable) almost product property leads to the fact that the statement could hold for general invariant measures. With Theorem \ref{uni sep}, mimicking the proof of Proposition 2.3 in \cite{Sulli2}, one can obtain the following result. (We omit the proof since it is quite similar.)
\end{rem}

\begin{thm}\label{subset}
Let $f: M\rightarrow M$ be a $C^1$-smooth partially hyperbolic diffeomorphism. If $(M, f)$ satisfies unstable almost product property for some blow up function $g$. Then for any $\kappa>0$, there exists a $\rho^*>0$ and an $\epsilon^*>0$ satisfying that for any $\mu\in \mathcal{M}_f(M)$, any neighborhood $F\subseteq \mathcal{M}(M)$ of $\mu$, and any $\delta>0$, there is a closed $f$-invariant $M'\subseteq M$ satisfying the following properties:

1) There exists a positive integer $n'$ such that $\mathcal{E}_n(y)\in F$ for any $n\geq n'$ and any $y\in M'.$

2) There is a subset $C\subseteq M$ and a positive integer $n''$ such that for any $x\in C$ and any $n\geq n''$, there exists an $(\rho^*, n, \epsilon^*)\,u$-separated set of $\overline{W^u(x, \delta)}\cap M'$ whose cardinality is greater than or equal to $e^{n(h^u_{\mu}(f)-\kappa)}.$
\end{thm}

\begin{coro}\label{un dense}
Let $f: M\rightarrow M$ be a $C^1$-smooth partially hyperbolic diffeomorphism. If $(M, f)$ satisfies unstable almost product property for some blow up function $g$, then ergodic measures of $(M, f)$ are unstable entropy dense.
\end{coro}
\begin{proof}
Take any $\mu\in\mathcal{M}_f(M)$, any neighborhood $F$ of $\mu$, and any $h^*<h^u_{\mu}(f)$, pick up a smaller neighborhood $F'\ni \mu$ with $\overline{F'}\subseteq F$. Apply Theorem \ref{subset} for $F'$, there exists a closed $f$-invariant subset $M'$ of $M$ such that for any $y\in M'$ and any $n\in\mathbb{N}$ large enough, one has $\mathcal{E}_n(y)\in F'$; moreover, $h^u_{UC}(f, M')>h^*$. By Proposition \ref{VP}, there is an ergodic measure $\omega$ on $M$ such that $\omega(M')=1$ and $h^u_{\omega}(f)> h^*$. Let $y\in M'$ be a generic point of $\omega$. Since $\{\mathcal{E}_n(y)\}_n$ converges to $\omega$ and $\mathcal{E}_n(y)\in F'$ for $n$ large enough, one has $\omega\in F$.
\end{proof}

\begin{thm}\label{less}
Let $f:M\rightarrow M$ be a $C^1$-smooth partially hyperbolic diffeomorphism  and $\mathcal{G}=\{\log g_n\}^{\infty}_{n=1}$ be a sequence of sub-additive potentials of $f$.

1) Let $K\subseteq\mathcal{M}_f(M)$ be a closed set, and let $$^{K}G:=\{x\in M\mid \{\mathcal{E}_n(x)\}_n\,\text{has a limit point in}\,K\}.$$ Then $$P^u_{B}(f, \mathcal{G},\,^{K}G)\leq \sup\limits_{\mu\in K}\{h^u_{\mu}(f)+\mathcal{G}_{*}(\mu)\}.$$
In particular, if $K$ consists of a single point $\mu\in\mathcal{M}_f(M)$, then $$P^u_B(f, \mathcal{G}, G_{\mu})\leq h^u_{\mu}(f)+\mathcal{G}_{*}(\mu).$$

2) Let $K\subseteq\mathcal{M}_f(M)$ be non-empty, connected, and compact, then $$P^u_{B}(f, \mathcal{G}, G_K)\leq\inf\limits_{\mu\in K}\{h^u_{\mu}(f)+\mathcal{G}_{*}(\mu)\}.$$

\end{thm}

\begin{proof} Note that the saturated set $G_K$ is a subset of $^{K}G$, and moreover $G_K\subseteq ^{\{\mu\}}\!G$ for any $\mu\in K$, so it suffices to show 1).

Suppose $s=\sup\{h^u_{\mu}(f)+\mathcal{G}_{*}(\mu)\mid\mu\in K\}<\infty$, otherwise it is already done. For any $\gamma>0$, set $s'=s+2\gamma$. We shall show $$P^u_{B}(f, \mathcal{G},\,^{K}G)\leq s'.$$ By Definition \ref{Bowen pres}, we need to show that for any $\epsilon>0$, any $x\in M$, and any $\delta>0$, one has $$m^u(\mathcal{G}, s', \epsilon,\,^{K}G, \overline{W^u(x, \delta)})=0.$$

By the proof of Proposition 3.4 in \cite{Zhang2}, we have $$\inf\limits_{F\ni \mu}\limsup\limits_{n\to \infty}\dfrac{1}{n}\log P^u(F, \mathcal{G}, n, \epsilon, x, \delta)\leq h^u_{\mu}(f)+\mathcal{G}_{*}(\mu),$$ where $$\begin{aligned}
&P^u(F, \mathcal{G}, n, \epsilon, x, \delta)\\
=&\sup\{\sum\limits_{y\in E}g_n(y)\mid E\;\text{is an}\;(n, \epsilon) \,u\text{-separated subset of}\;\overline{W^u(x, \delta)}\cap M_{n,F}\}.
\end{aligned}
$$ So there exists a neighborhood of $\mu$, say $F(\mu, \epsilon)$, and a positive integer $M$ such that if $n>M$, then $$\dfrac{1}{n}\log P^u(F(\mu, \epsilon), \mathcal{G}, n, \epsilon, x, \delta)\leq h^u_{\mu}(f)+\mathcal{G}_{*}(\mu)+\gamma.$$

One can construct in a similar way a maximal $(n, \epsilon) \,u$-separated subset $\Gamma$ as in Proposition 2.1 of \cite{Zhang}, it is also an $(n, \epsilon)\, u$-spanning set, then for any $n>M$, one has
$$
\begin{aligned}
M^u(\mathcal{G}, s', n, \epsilon, M_{n, F(\mu, \epsilon)}, \overline{W^u(x, \delta)})\leq&\sum\limits_{x_i\in\Gamma}\sup\limits_{y\in B^u_n(x_i, \epsilon)} g_n(y)e^{-s'n}\\
\leq&P^u(F(\mu, \epsilon), \mathcal{G}, n, \epsilon, x, \delta)e^{-s'n}\\
\leq&e^{n(s+\gamma)}e^{-s'n}=e^{-n\gamma}.
\end{aligned}
$$  By compactness of $K$, we can find a finite open cover of $K$ which consists of $F(\mu_i, \epsilon), 1\leq i\leq m_{\epsilon}$ with $\mu_i\in K$. Set $$A_N=\bigcup\limits_{n\geq N}\bigcup_{i=1}^{m_{\epsilon}}M_{n, F(\mu_i, \epsilon)}.$$ Then for any $M\geq\max\limits_{1\leq i\leq m_{\epsilon}}M(F(\mu_i, \epsilon)),$ one has $$M^u(\mathcal{G}, s', M, \epsilon,\,^{K}G, \overline{W^u(x, \delta)})\leq m_{\epsilon}\sum_{n\geq M}e^{-\gamma n},$$ so $$m^u(\mathcal{G}, s', \epsilon,\,^{K}G, \overline{W^u(x, \delta)})=0.$$ Hence $$P^u_{B}(f, \mathcal{G},\,^{K}G)\leq \sup\limits_{\mu\in K}\{h^u_{\mu}(f)+\mathcal{G}_{*}(\mu)\}.$$
\end{proof}

\textbf{Now we proceed to prove Theorem \ref{main1}.}
\begin{proof}
The main idea is that for any $\eta>0$ we shall construct a closed subset $Z$ of $G_{K}$ such that $$P^u_B(f, \mathcal{G}, Z)\geq \inf\{h^u_{\mu}(f)+\mathcal{G}_{*}(\mu)\mid \mu\in K\}-\eta.$$

First we present the construction of the subset $Z$.
By the compactness of $K$, for any $\epsilon>0$, one can find a finite $\epsilon$-net of $K$, say $\mu_1, \mu_2, \cdots, \mu_N$. Since $K$ is connected, one can repeat some $\mu_i, 1\leq i\leq N$ if necessary, to get the new ones $\{\nu_i\}^{N'}_{i=1}$ such that they  still form a $\epsilon$-net of $K$ and $d(\nu_i, \nu_{i+1})<2\epsilon$. Continue this argument even further, one can find a sequence of measures, say $\{\nu'_i\}^\infty_{i=1}$, such that the closure of $\{\nu'_j\mid j>n\}$ equals $K$ for any positive integer $n$ and $\lim\limits_{j\to\infty}d(\nu'_j, \nu'_{j+1})=0$.

Set $$S^*=\inf\{h^u_{\mu}(f)+\mathcal{G}_{*}(\mu)\mid \mu\in K\}-\eta.$$
For the sequence $\{\nu'_i\}_{i\in\mathbb{N}}$, we will construct a subset $Z$ such that for any $x\in Z$, one has $\{\mathcal{E}_n(x)\}$ has the same limit point set as $\{\nu'_i\}_{i\in\mathbb{N}}$ and $P^u_B(f, \mathcal{G}, Z)\geq S^*.$

Suppose $\{\zeta_k\}_k$ and $\{\epsilon_k\}_k$ are two decreasing sequences with $\lim\limits_{k\to\infty}\zeta_k=\lim\limits_{k\to\infty}\epsilon_k=0.$ By Corollary \ref{inv}, for the $\eta>0$ above, one can find $\rho^*>0$ and $\epsilon^*>0$ such that for any neighborhood $F\subset\mathcal{M}(M)$ of $\mu\in K$ and any $\delta>0$, there exists a subset $C$ and a positive integer $n^*_{F, \mu, \eta}$ satisfying for any $x\in C$, if $n\geq n^*_{F, \mu, \eta}$, then \begin{align}\label{(4.1)}P^u(F, \mathcal{G}, n, \epsilon^*,  \rho^*, x, \delta)\geq e^{n\left(h^u_{\mu}(f)+\mathcal{G}_{*}(\mu)-\eta\right)}. \end{align}  Apply this to the sequence of measures $\{\nu'_k\}_k$ and neighborhoods $B(\nu'_k, \zeta_k)$, there exists a subset $C_k$ and positive integer $n_k$ such that for any $x_k\in C_k$, if $n\geq n_k$, there is a $(\rho^*, n, \epsilon^*)\,u$-separated set $\Gamma_k$ of $M_{n_{k}, B(\nu'_k, \zeta_k)}\cap \overline{W^u(x_k, \delta)}$ satisfying $$E(\Gamma_k):=\sum_{y\in\Gamma_k}g_{n_k}(y)\geq e^{n_kS^*} \quad(*).$$
Furthermore, we may further assume $n_k$ satisfies $$\rho^*n_k>2g(n_k)+1\;\text{and}\;\dfrac{g(n_k)}{n_k}\leq\epsilon_k, \quad(1).$$
By Lemma \ref{measure-nbh} and (1), for any $x\in\Gamma_k$ and $y\in B^u_{n_k}(g, x, \epsilon_k)$, one has $$\mathcal{E}_{n_k}(y)\in B(\nu'_k, \zeta_k+2\epsilon_k), \quad(2).$$

Now we use the orbit segment $\{x, fx, \cdots, f^{n_k-1}x\}(x\in \Gamma_k)$ to give the construction of $Z$. Let $\{N_k\}_k$ be an increasing sequence of positive integers such that $$n_{k+1}\leq\zeta_k\sum^{k}_{i=1}n_iN_i  \;\text{and}\, \sum^{k-1}_{i=1}n_iN_i\leq \zeta_k\sum^{k}_{i=1}n_iN_i, \quad(3).$$ Define new sequences as follows: $$n'_i:=n_k, \epsilon'_i:=\epsilon_k, x'_i=x_k, C'_i=C_k,\,\text{and}\, \Gamma'_i:=\Gamma_k$$ if $i=N_1+N_2+\cdots+N_{k-1}+q$ with $1\leq q\leq N_k-1.$ Set $$Z_k:=\bigcap^{k}\limits_{i=1}\bigcup\limits_{x'_i\in\Gamma'_i}f^{-M_{i-1}}B^u_{n_i}(g, x'_i, \epsilon'_i),$$ where $M_i=\sum^{i}_{l=1}n'_l$. By the definition of unstable almost product property for the blow up function $g$, $Z_k$ is a non-empty closed set. We can label each $Z_k$ by developing this formula from the branches of a labeled tree of height $k$. A branch is labeled by $(x'_1, \cdots, x'_k)$ with $x'_i\in \Gamma'_i$.

Set $Z:=\bigcap\limits_{k\geq 1}Z_k$. Next we are going to obtain several conclusions about $Z$.

\textit{\textbf{Claim I.}} Suppose we have two different points $x_i, y_i\in \Gamma'_i$. If $x\in B^u_{n'_i}(g, x_i, \epsilon'_i)$ and $y\in B^u_{n'_i}(g, y_i, \epsilon'_i)$, then $$\max\limits_{1\leq m\leq n_i-1}d^u(f^mx, f^my)>\epsilon^*/2.$$

Indeed $x_i, y_i\in \Gamma'_i$ amounts to say that they are $(\rho^*, n'_i, \epsilon^*)\, u$-separated, while based on (1), there is an $m\in \{1, \cdots, n'_i-1\}$ such that $$d^u(f^mx_i, f^my_i)>4\epsilon^*, d^u(f^mx_i, f^mx)\leq \epsilon'_i, \, \text{and}\, d^u(f^my_i, f^my)\leq \epsilon'_i.$$ So $$d^u(f^mx, f^my)\geq d^u(f^mx_i, f^my_i)-d^u(f^mx_i, f^mx)-d^u(f^my, f^my_i)>\epsilon^*/2.$$

\textit{\textbf{Claim II.}} $Z$ is a non-empty closed set, which can be written as a disjoint union of non-empty closed sets $Z(x_1, x_2,  \cdots)$, which is labeled by $(x_1, x_2, \cdots)$ with each $x_i\in \Gamma'_i$. Moreover, $Z\subseteq W^u(x'_1)$.

\textit{\textbf{Claim III.}} $Z\subseteq G_K$.

In fact, define a stretched sequence of measures $\{\nu''_q\}_q$ by $$\nu''_q=\nu'_k\;\text{if}\, \sum^{k-1}_{i=1}n_iN_i+1\leq q\leq \sum^k_{i=1}n_iN_i.\quad\quad(4)$$ It is obvious that the sequence $\{\nu''_q\}_q$ has the same limit point set as $\{\nu'_k\}_k$. For any $y\in Z$, if $\lim\limits_{n}d(\mathcal{E}_n(y), \nu''_n)=0,$ then $\{\mathcal{E}_n(y)\}_n$ and $\{\nu''_n\}_n$ have the same limit point set. By (3) and the construction of $\{\nu''_n\}$, one only needs to show $$\lim\limits_{k\to\infty}d(\mathcal{E}_{M_k}(y), \nu''_{M_k})=0,$$ where $M_k=\sum^k\limits_{l=1}n'_l.$

Suppose $$\sum^i\limits_{l=1}n_lN_l<M_k\leq \sum^{i+1}\limits_{l=1}n_lN_l,$$ and then $\nu''_{M_k}=\nu'_{i+1}$. By Lemma \ref{measure-nbh}, (1) and (4), one has $$
\begin{aligned}
d(\mathcal{E}_{M_k}(y), \nu''_{M_k})\leq&\dfrac{\sum^{i-1}\limits_{l=1}n_lN_l}{M_k}d(\mathcal{E}_{\sum^{i-1}\limits_{l=1}n_lN_l}(y), \nu''_{M_k})\\
+&\dfrac{n_iN_i}{M_k}(\zeta_i+2\epsilon_i+d(\nu'_i, \nu'_{i+1}))\\
+&\dfrac{M_k-\sum^i\limits_{l=1}n_lN_l}{M_k}(\zeta_{i+1}+2\epsilon_{i+1}).
\end{aligned}$$ Since $\lim\limits_{i}\zeta_i=0$, $\lim\limits_{i}\epsilon_i=0$, and $\lim\limits_{i}d(\nu'_i, \nu'_{i+1})=0$, we are done.

\textit{\textbf{Claim IV.}} $P^u_B(f, \mathcal{G}, Z)\geq S^*.$

For any $S<S^*$ and any $x\in C'_1$ with $Z\subseteq \overline{W^u(x'_1)}$, we shall show $$M^u(\mathcal{G}, S, N, \epsilon, Z, \overline{W^u(x'_1, \delta)})>0, \forall\, \delta>0\,\text{small enough}.$$ Since $Z\cap \overline{W^u(x, \delta)}$ is compact, we consider an arbitrary finite cover $\Gamma_F=\{B^u_{n_i}(x_i, \epsilon)\}^k_{i=1}$ with $n_i\geq N$. For each such a finite cover $\mathcal{C}$, we can derive a new one $\mathcal{C}'$ by replacing every $B^u_{n_i}(x_i, \epsilon)$ by $B^u_{M_i}(x_i, \epsilon)$ if $M_i\leq n_i< M_{i+1}$. Note that $\lim\limits_{n\to\infty}M_{n+1}/M_n=1$. Then $$M^u(\mathcal{G}, S, N, \epsilon, Z, \overline{W^u(x, \delta)})\geq \inf\limits_{\Gamma_F}\{\sum\limits_{B^u_{M_i}(x_i, \epsilon)\in \mathcal{C}'}(\sup\limits_{y\in B^u_{M_{i+1}}(x_i, \epsilon)}g_{n_i}(y))e^{-sM_{i+1}}\}.$$

For such a cover $\mathcal{C}'$, let $t$ be the largest value of $p$ for which there exists a $B^u_{M_p}(x,\epsilon)\in\mathcal{C}'$. Set $W_k=\prod\limits^k_{i=1}\Gamma'_i, k=1,...,t,$ and $\widetilde{W_t}=\bigcup^t\limits_{k=1}W_k.$ By \textit{\textbf{Claim I.}}, for any $y\in B^u_{M_p}(z, \epsilon)\cap Z$, there is a unique word $w_p\in W_p$ corresponding to $y$ in the following sense: $$w_p=(x_1, x_2, \cdots, x_{M_p})\Rightarrow y\in\bigcap^{M_p}\limits_{j=1}f^{-M_{j-1}}B^u_{n'_j}(g, x_j, \epsilon'_j),$$ where $M_0=0$ and $M_i=n'_1+\cdots+n'_i$. For $1\leq i\leq k$, we say a word $v\in W_i$ is a prefix of $w\in W_k$ if the first $i$ entries of $w$ coincide with $v$. 
For $W_k (k\geq 1)$, define $$E_n(W_k)=\prod\limits_{i=1}^k(\sum\limits_{y\in\Gamma'_i} g_{n'_i}(y)).$$

Suppose $W\subseteq \widetilde{W_t}$ contains a prefix of each word in $W_t$. Then it is obvious that $$\sum\limits_{k=1}^t|W\cap W_k|\dfrac{E_n(W_t)}{E_n(W_k)}\geq E_n(W_t),$$ and so $$\sum_{k=1}^t|W\cap W_k|\dfrac{1}{E_n(W_k)}\geq 1.$$ Since $\mathcal{C}'$ is a cover of $Z\cap \overline{W^u(x, \delta)}$, each word in $W_t$ has a prefix associated with some $B^u_{M_i}(x_i, \epsilon)\in \mathcal{C}'$ and $E_n(W_k)\geq e^{M_kS^*}$ based on $(*)$ before. Hence $$1\leq \sum_{k=1}^t|W\cap W_k|\dfrac{1}{E_n(W_k)}=\sum\limits_{B^u_{M_k}(x_i, \epsilon)\in \mathcal{C}'}\dfrac{1}{E_n(W_k)}\leq \sum\limits_{B^u_{M_k}(x_i, \epsilon)\in \mathcal{C}'}e^{-M_kS^*}.$$ Since $\lim\limits_{n\to\infty}M_{n+1}/M_n=1$, for any $0<\gamma\leq \frac{S^*-S}{S}$, there exists an $N$ such that $$|\dfrac{M_{n+1}}{M_n}-1|\leq \gamma, \forall\, n\geq N,$$ which implies $SM_{n+1}\leq S^*M_n$.

Set $m=\min\{\sup\limits_{y\in B^u_{n_i}(x_i, \epsilon)}g_{n_i}(y)\mid B^u_{M_i}(x_i, \epsilon)\in \mathcal{C}'\},$ then $$
\begin{aligned}
&\sum\limits_{B^u_{M_i}(x_i, \epsilon)\in \mathcal{C}'}(\sup\limits_{y\in B^u_{M_{i+1}}(x_i, \epsilon)}g_{n_i}(y))e^{-SM_{i+1}}\\
&\geq m\sum\limits_{B^u_{M_i}(x_i, \epsilon)\in \mathcal{C}'}e^{-SM_{i+1}}\geq m\sum\limits_{B^u_{M_i}(x_i, \epsilon)\in \mathcal{C}'}e^{-S^*M_{i}}\geq m>0.\\
\end{aligned}$$
So $M^u(\mathcal{G}, S, N, \epsilon, Z, \overline{W^u(x, \delta)})\geq m>0$ and $P^u_B(f, \mathcal{G}, Z)\geq S^*.$
\end{proof}


We present an application of Theorem \ref{main1} here.
\begin{prop}
Let $f:M\rightarrow M$ be $C^1$-smooth partially hyperbolic diffeomorphism, and $\mathcal{G}=\{\log g_n\}_{n\geq 1}$ be a sequence of sub-additive potentials of $f$ on M. Suppose $P^u_B(f, \mathcal{G}, G_{\mu})=h^u_{\mu}(f)+\mathcal{G}_{*}(\mu)$ for any $\mu\in \mathcal{M}_f(M)$. For any real valued continuous function $\phi\in C(M)$ and any real number $a$, set $$K_a=\{x\in M\mid \lim\limits_{n\to\infty}\frac{1}{n}\sum\limits_{k=0}^{n-1}\phi(f^kx)=a\},$$ then $$P^u_B(f, \mathcal{G}, K_a)=\sup\{h^u_{\mu}(f)+\mathcal{G}_*(\mu)\mid \mu\in \mathcal{M}_f(M)\,\text{and}\, \int_{M}\phi\,d\mu=a\}.$$
\end{prop}
\begin{proof}
Set $\mathcal{M}(a)=\{\mu\in\mathcal{M}_f(M)\mid \int_M\phi\, d\mu=a\}$, then it is closed. Let $$G^{\mathcal{M}(a)}=\{x\in M\mid \{\mathcal{E}_n(x)\}_n \,\text{has all its limit points in}\,\mathcal{M}(a)\},$$ so $K_a=G^{\mathcal{M}(a)}$.

On one hand, note that $G^{\mathcal{M}(a)}\subseteq ^{\mathcal{M}(a)}\!\!G$, so by $1)$ of Theorem \ref{less}, one has $$P^u_B(f, \mathcal{G}, K_a)\leq\sup\{h^u_{\mu}(f)+\mathcal{G}_*(\mu)\mid \mu\in \mathcal{M}_f(M)\,\text{and}\, \int_{M}\phi\,d\mu=a\}.$$ On the other hand, for any $\mu\in \mathcal{M}(a)$, one has $G_{\mu}\subseteq G^{\mathcal{M}(a)}$. By assumption we know $$P^u_B(f, \mathcal{G}, G_{\mu})=h^u_{\mu}(f)+\mathcal{G}_{*}(\mu),$$ then $$P^u_B(f, \mathcal{G}, K_a)\geq\sup\{h^u_{\mu}(f)+\mathcal{G}_*(\mu)\mid \mu\in \mathcal{M}_f(M)\,\text{and}\, \int_{M}\phi\,d\mu=a\}.$$
\end{proof}

In particular, for the corresponding entropy, we have: 
\begin{coro}
Let $f:M\rightarrow M$ be $C^1$-smooth partially hyperbolic diffeomorphism. Suppose $h^u_B(f, G_{\mu})=h^u_{\mu}(f)$ for any $\mu\in \mathcal{M}_f(M)$. For any real valued continuous function $\phi\in C(M)$ and any real number $a$, consider the level set $K_a$ as before, then one has $$h^u_B(f, K_a)=\sup\{h^u_{\mu}(f)\mid \mu\in \mathcal{M}_f(M)\,\text{and}\, \int_{M}\phi\,d\mu=a\}.$$ 
\end{coro}

\section{Unstable Capacity Pressure}

\textbf{Now we proceed to prove Theorem \ref{main2}.}

\begin{proof}
Since $G_{K}\subseteq M$, one has $$\underline{CP}^u(f, \mathcal{G}, G_K)\leq P^{u}(f,\mathcal{G}).$$ We only need to show $$\underline{CP}^u(f, \mathcal{G}, G_K)\geq P^{u}(f,\mathcal{G}).$$ First, we construct a closed subset $Z$ of $G_K$ such that $\underline{CP}^u(f, \mathcal{G}, Z)$ can arbitrarily close to $P^{u}(f,\mathcal{G})$. By the variational principle for unstable topological pressure in \cite{Zhang}, for any $\eta>0$, there exists $\mu\in\mathcal{M}^{e}_{f}(M)$ such that $h^{u}_{\mu}(f)+\mathcal{G}_{*}(\mu)>P^{u}(f,\mathcal{G})-\dfrac{\eta}{2}$. For $\eta$ above, by (1), there exists $\epsilon_0>0$ and $N\in\mathbb{N}$ such that for any $\epsilon<\epsilon_0$ and $n\geq N$, one has $$\dfrac{\log \gamma_{n}(\mathcal{G},\epsilon)}{n}\leq\dfrac{\eta}{2}.$$ On the other hand, By Proposition \ref{first half for ergodic}, for the $\eta$, one can find $\rho^*>0$ and $\epsilon^*>0$ such that for any $\mu\in\mathcal{M}^e_{f}(M)$, any neighborhood $F$ of $\mu$ in $\mathcal{M}(M)$, and any $\delta>0$, there exists a subset $C$ of $M$ with $\mu(C)>0$ and a positive integer $n^*_{F, \mu, \eta}$ such that for any $x\in C$, if $\mathcal{N}\geq n^*_{F, \mu, \eta}$, one can find a $(\rho^*, \mathcal{N}, 3\epsilon^*)$ $u$-separated set $\Gamma_{\mathcal{N}}=\{z_{1}, z_{2},\cdots, z_{r}\}$ of $\overline{W^{u}(x,\delta)}\bigcap M_{n,F}$ such that $$\sum\limits_{z_{i}\in \Gamma_{\mathcal{N}}}g_{\mathcal{N}}(z_{i})\geq e^{\mathcal{N}\left(h^u_{\mu}(f)+\mathcal{G}_{*}(\mu)-\frac{\eta}{2}\right)}.$$ Take $\mathcal{N}$ large enough such that $$\mathcal{N}>N,\;\text{and}\;\dfrac{g(\mathcal{N})}{\mathcal{N}}<\dfrac{\rho^*}{2}.\;\;\;(2)$$ Suppose $\{\zeta_{k}\}_{k}$ and $\{\epsilon_{k}\}_{k}$ are two strictly decreasing sequences satisfying $$\lim\limits_{k\rightarrow\infty}\zeta_{k}=\lim\limits_{k\rightarrow\infty}\epsilon_{k}=0.$$ By the proof of Theorem \ref{main1} in Section 3, there exists a sequence of measures $\{\alpha_{1}, \alpha_{2},\cdots\}\subseteq K$ satisfying $\overline{\{\alpha_j \mid j\in \mathbb{N}, j\geq n\}}=K,$ for any $n\in\mathbb{N}$ and $\lim\limits_{j\rightarrow\infty}(\alpha_{j},\alpha_{j+1})=0$. For any fixed $z_0\in G_K\bigcap\overline{W^{u}(x,\delta)}$, there exists $n_k\in\mathbb{N}$ such that $\dfrac{1}{n_k}\sum\limits^{n_k-1}_{j=0}\delta_{f^{i}(z_0)}\in B(\alpha_{k},\zeta_{k})$. One can suppose $n_k$ is large enough such that $n_{k}> m(\epsilon_{k})$ and $\dfrac{g(n_k)}{n_k}<\epsilon_k$. Take a strictly increasing integer sequence  $\{N_k\}_k$ such that $$n_{k+1}\leq\zeta_k(\sum^{k}_{i=1}n_iN_i+\mathcal{N})  \;\text{and}\, \sum^{k-1}_{i=1}n_iN_i+\mathcal{N}\leq \zeta_k(\sum^{k}_{i=1}n_iN_i+\mathcal{N}), \quad(3).$$ Define the sequences $\{n'_{j}\}$, $\{\epsilon'_{j}\}$ and $\{\Gamma'_{j}\}$ inductively, by setting,  $$n'_0:=\mathcal{N},\,\epsilon'_0:=\epsilon^*,\, \Gamma'_0:=\Gamma_{\mathcal{N}},$$ and for $$j=\sum\limits^{k-1}_{i=1}N_{i}+q,\,n'_j:=n_{k},\,\epsilon'_j:=\epsilon_{k},\, \Gamma'_j:=\{z_0\}.$$ For any $s\in\mathbb{N}^{+}$, let
$$G^{\mathcal{N}}_{s}:=\bigcap\limits^{s}_{j=0}\left(\bigcup\limits_{x_{j}\in\Gamma'_{j}}f^{-M_{j-1}}B^{u}_{n_{j}}(g;x_{j},\epsilon'_{j})\right),$$ where $M_j:=\sum\limits^{j}_{l=0}n'_{l}$. By the unstable almost product property, $G^{\mathcal{N}}_{s}$ is a non-empty closed set.

Set $G^{\mathcal{N}}:=\bigcap\limits_{s\geq1}G^{\mathcal{N}}_{s}$. Then $G^{\mathcal{N}}$ is also non-empty and closed. With the similar discussion as the proof of Theorem \ref{main1} of \textit{\textbf{Claim III}}, one can show that $G^{\mathcal{N}}\subseteq G_K$.

We claim that there exists a set $Y\subseteq G^{\mathcal{N}}$, which is an $(\mathcal{N}, \epsilon^*)$ $u$-separated set satisfying $$\text{Card}(Y)=\text{Card}(\Gamma_{\mathcal{N}})\;\text{and}\;\sum\limits_{y\in Y}g_{\mathcal{N}}(y)\geq e^{\mathcal{N}\left(h^u_{\mu}(f)+\mathcal{G}_{*}(\mu)-\frac{\eta}{2}\right)}.$$

The proof of the claim: For any $z_{i}\in \Gamma_{\mathcal{N}}$, set $$G^{(i)}:=B^{u}_{\mathcal{N}}(g;z_i,\epsilon^*)\bigcap\left(\bigcap\limits_{j\geq1}f^{-M_{j-1}}B^{u}_{n'_j}(g;z_0,\epsilon'_j)\right).$$
Then $G^{\mathcal{N}}:=\bigcup\limits^{r}_{i=1}G^{(i)}$. By the unstable almost product property, $G^{(i)}$ is a non-empty and closed set. For $i=1,2,\cdots,r$, take $y_{i}\in G^{(i)}$, in which way we get a set $Y=\{y_{1}, y_{2},\cdots,y_{r}\}$. For any tow different points $z_{u},z_{v}\in\Gamma_{\mathcal{N}}$, they are $(\rho^*,\mathcal{N},3\epsilon^*)\,u$-separated. By $(2)$ there exists $m\in \{0,\cdots, \mathcal{N}-1\}$ such that $$d^u(f^m z_u, f^m z_v)\geq3\epsilon^*,\;\;d^u(f^m z_u, f^m y_u)\leq\epsilon^*,\;\text{and}\;d^u(f^m z_v, f^m y_v)\leq\epsilon^*.$$ So
$$d^u(f^m y_u, f^m y_v)\geq d^u(f^m z_u, f^m z_v)-d^u(f^m z_u, f^m y_u)-d^u(f^m z_v, f^m y_v)\geq\epsilon^*.$$
So the set $Y$ is an $(\mathcal{N},\epsilon^*)\,u$-separated set of $G^{(\mathcal{N})}\cap\overline{W^{u}(x,\delta)}$.
Moreover, $$\begin{aligned}
&\Lambda^u(\mathcal{G}, n, \dfrac{\epsilon^*}{2}, G^{(\mathcal{N})}, \overline{W^u(x, \delta)})\\&:=\inf\limits_{\Gamma}\{\sum\limits_{i}\sup\limits_{y\in B^u_{n}(x_i, \dfrac{\epsilon^*}{2})}g_n(y)\}&\geq \sum\limits_{y_{i}\in Y}g_n(y_i)\\
&\geq\sum\limits_{z_{i}\in \Gamma_{\mathcal{N}}}e^{\log g_n(z_i)-\frac{\mathcal{N}\eta}{2}}\;\;\text{by}\;(1)\\
&\geq e^{\mathcal{N}(\mathcal{G}_{*}(\mu)+h^{u}_{\mu}(f)-\frac{\eta}{2})}.
\end{aligned}$$
Then one has $$\underline{CP}^{u}(f,\mathcal{G},G_K)\geq \underline{CP}^{u}(f,\mathcal{G},G^{\mathcal{N}})\geq \mathcal{G}_{*}(\mu)+h^{u}_{\mu}(f)-\frac{\eta}{2})\geq P^{u}(f,\mathcal{G})-\eta.$$
By the arbitrariness of $\eta$, one gets  $$\underline{CP}^u(f, \mathcal{G}, G_K)\geq P^{u}(f,\mathcal{G}).$$
\end{proof}

\proof[Acknowledgements] The first author is supported by the NSFC (National Science Foundation of China) grant with grant No.\,11501066, and the grant from the Department of Education of Chongqing City with contract No.\,KJQN202100722 in Chongqing Jiaotong University. She is also sponsored by Natural Science Foundation of Chongqing, China, with grant No.\ cstc2021jcyj-msxmX1045.

The second author is sponsored by Natural Science Foundation of Chongqing, China, with grant No.\ cstc2021jcyj-msxmX1042,
as well as Chongqing Key Laboratory of Analytic Mathematics and Applications in Chongqing University.

\end{document}